\documentclass{tran-l}

\newtheorem{theorem}{Theorem}[section]
\newtheorem{lemma}[theorem]{Lemma}
\newtheorem{definition}[theorem]{Definition}
\newtheorem{proposition}[theorem]{Proposition}

\usepackage[left=3cm,right=3cm,
    top=3cm,bottom=3cm,bindingoffset=0cm]{geometry}
\usepackage[utf8]{inputenc} % Required for inputting international characters
\usepackage[T1]{fontenc} % Output font encoding for international characters

\usepackage{amsmath}
\usepackage{mathtools}
\usepackage{amsfonts}
\usepackage{amsthm}
\usepackage{amssymb}
\usepackage{tikz-cd} 

\usepackage{yfonts}

\DeclareMathOperator{\Sec}{Sec}
\DeclareMathOperator{\spann}{span}

\DeclareMathOperator{\Symm}{Sym}

\DeclareMathOperator{\imm}{Im}

\theoremstyle{definition}

\theoremstyle{remark}
\newtheorem{remark}[theorem]{Remark}

\numberwithin{equation}{section}

\begin{document}

\title{Geometry of elliptic normal curves of degree 6}

\author{Anatoli Shatsila}
\address{Institute of Mathematics, Jagiellonian University in Kraków, Poland}
\curraddr{}
\email{anatoli.shatsila@student.uj.edu.pl}
\thanks{}

%    General info
\subjclass[2022]{14H52}

\date{}

\dedicatory{}

\keywords{}

\begin{abstract}

In our work we focus on the geometry of elliptic normal curves of degree 6 embedded in $\mathbb{P}^5$.  
We determine the space of quadric hypersurfaces through an elliptic normal curve of degree 6 and find the explicit equations of generators of $I(\Sec(C_6))$.
We study the images $C_p$ and $C_{pq}$ of a sextic $C_6$ under the projection from a general point $P \in \mathbb{P}^5$ and a general line $\overline{PQ} \subset \mathbb{P}^5$. In particular, we show that $C_p$ is $k$-normal for all $k \geq 2$ and $I(C_p)$ is generated by three homogeneous polynomials of degree 2 and two homogeneous polynomials of degree 3. We then show that $C_{pq}$ is $k$-normal for all $k \geq 3$ and $I(C_{pq})$ is generated by two homogeneous polynomials of degree 3 and three homogeneous polynomials of degree 4.   

\end{abstract}

\maketitle
%% The correct journal style for \specialsection is all uppercase; a known bug
%% in amsart.cls prevents this, so input must be uppercase until it is fixed.
%\specialsection*{This is a Special Section Head}
%%%%%%%%%%%%%%%%%%%%%%%%%%%%%%%%%%%%%%%%%%%%%%%%%%%%%%%%%%%%%%%%%%%%%%%%
%\footnote{}%
%%%%%%%%%%%%%%%%%%%%%%%%%%%%%%%%%%%%%%%%%%%%%%%%%%%%%%%%%%%%%%%%%%%%%%%%
.

\section{Introduction}

Elliptic normal curves, which are elliptic curves embedded in $\mathbb{P}^{n - 1}$ as degree $n$ curves and not contained in any hyperplane, have been studied since the late 19-th century (\cite{Bian, Klein, Hur}). The cases $n = 3$ and $n = 4$ are classical and well-known.
%The case $n = 3$ naturally arises in the study of elliptic integrals $\int \frac{dt}{\sqrt{t(t-1)(t-\lambda)}}$. This theory is classical and well-known (see \cite[Chapter VI]{Silv} and \cite{Chan}). In particular, any elliptic curve in $\mathbb{P}^2$ is given by the Weierstrass equation $Y^2Z = X^3 + AXZ^2 + BZ^3$ for some $A, B \in \mathbb{C}$. 
% In the case $n = 4$ it is known that any such curve is the complete intersection of two quadric surfaces.
The case $n = 5$ has been considered in \cite{Hul}. In the monograph K. Hulek studies the interrelation between the Horrocks-Mumford vector bundle and the normal bundle of elliptic curves of degree 5. Many beautiful geometric constructions and configurations arise in this monograph. Inspired by the work of Hulek, we explore the geometry of several geometric objects related to normal elliptic curves of degree $6$ embedded in $\mathbb{P}^5$. 

In section 2 we describe an embedding of elliptic curve as a normal curve of degree $n$ in $\mathbb{P}^{n - 1}$ by means of Weierstrass sigma-function. Under such embedding, translation by $n$-torsion points and $-1$ involution of the curve take a particularly simple form, which allows one to use the symmetries of the curve to study its geometry. 

In section 3 we study the action of the Heisenberg group $H_6$ on the space $H^0(\mathcal{O}_{\mathbb{P}^5}(2))$ of quadric hypersurfaces containing the curve $C_6$ of degree 6. Careful analysis of the decomposition of $H^0(\mathcal{O}_{\mathbb{P}^5}(2))$ allows us to find the basis of this space explicitly. We use this result to find the generators of $I(\Sec{C_6})$. In \cite{KK} Kaneko and Kuwata consider an embedding defined by theta functions and find the equations of quadrics using the Jacobi's identity. Although their method allows to compute the equations for any integer $n$, in case $n = 6$ our approach avoids the use of algebraic identities involving embedding functions, hence it is more explicit.

In the last section we study the images $C_p$ and $C_{pq}$ of $C_6$ under the projection from a general point and a general line. We prove the following theorem, which we consider as the main result of the paper: 

\begin{theorem}

The ideal $I(C_p)$ of the curve $C_p$ is generated by three polynomials of degree $2$ and two polynomials of degree 3.

The ideal $I(C_{pq})$ of the curve $C_{pq}$ is generated by two polynomials of degree $3$ and three polynomials of degree 4.

\end{theorem}

We prove this result in a few steps. Firstly, we use the $k$-normality property and Castelnuovo–Mumford regularity theory to find the dimensions of $h^0(\mathcal{I}_{C_p}(2))$ and $h^0(\mathcal{I}_{C_p}(3))$ and show that $I(C_p)$ is generated in degree 3. Then we consider a certain elementary geometric construction to show that quadric hypersurfaces, whose equations form a basis of $H^0(\mathcal{I}_{C_p}(2))$, intersect in a curve. This implies that there are no linear syzygies between them and allows us to finish the proof by counting the total contribution of the basis of $H^0(\mathcal{I}_{C_p}(2))$ in $H^0(\mathcal{I}_{C_p}(3))$. The proof of the result for $C_{pq}$ is analogous and even easier due to the fact that $h^0(\mathcal{I}_{C_{pq}}(3)) = 2$.

\section{Preliminaries}

We work over $\mathbb{C}$. By the Riemann-Roch Theorem, we can embed any elliptic curve $C$ as an elliptic normal curve of degree $n$ in $\mathbb{P}^{n - 1}$ via the complete linear system $|\mathcal{O}_C(n\cdot O)|$, where $O$ is the origin in $C$. 
Suppose that the embedding is given by $$\varphi: C = \mathbb{C} / (\mathbb{Z}\omega_1 + \mathbb{Z}\omega_2) \ni z \mapsto \varphi(z) = (X_0(z): X_1(z) : \ldots : X_{n - 1}(z)) \in \mathbb{P}^{n - 1}.$$ By $C_n$ we denote the image $\varphi(C)$. Recall the definition of the discrete Heisenberg group.
\begin{definition} 

Let $V = \mathbb{C}^n$ be a vector space with the standard basis $\{e_m\}_{m \in \mathbb{Z}/n\mathbb{Z}}$. Define $\sigma, \tau \in GL(V)$ by $$\sigma(e_m) :=e_{m-1}, \:\:\:\:\:\:\:\: \tau(e_m) := \varepsilon^m e_m,$$ where $\varepsilon = e^{\frac{2\pi i}{n}}$. 

The subgroup $H_n \subset GL(V)$ generated by $\sigma$ and $\tau$ is called the discrete Heisenberg group of dimension $n$. 

\end{definition}

We will consider embeddings under which $C_n$ admits an action of the Heisenberg group and is invariant under this action so that for any $z \in C$ both $\sigma(\varphi(z)) = (X_{n - 1}(z):X_0(z):\ldots:X_{n - 3}(z):X_{n - 2}(z))$ and $\tau(\varphi(z)) = (X_0(z):\varepsilon X_1(z): \ldots : \varepsilon^{n - 1}X_{n - 1}(z))$ are in $C_n$. In addition, we will assume the existence of the point $c \in C$ that satisfies $X_i(c) = 0 \Leftrightarrow i = 0$.

Note that such embeddings exist: the functions $\{x_m\}_{m \in \mathbb{Z}/n\mathbb{Z}}$ considered in \cite[Section I.2]{Hul} and defined using Weierstrass sigma-functions induce an immersion satisfying the above conditions (see \cite[Theorems I.2.3 and I.2.5]{Hul}). In this case we can take $c := \frac{\omega_1}{2} + \frac{\omega_2}{2n}$. For convenience, we will assume that $C$ is embedded by these functions, although all results (excluding the values of $\alpha, \beta$ and $\gamma$ in Section 3) and proofs are valid for any immersion with the mentioned properties.

Let us finish the section with three classical results.

\begin{lemma}[\cite{Hul}, Lemma IV.1.1]
\label{lemka}
Let $P_1, \ldots , P_k$ be $k$ different points on $C_n$. Then these points are linearly independent if $k \leq n - 1$, i.e. they span a subspace of dimension $k-1$.

\end{lemma}

\begin{proposition}[\cite{Hul}, Proposition IV.1.2]
\label{propka}
Every elliptic curve $C_n$ of degree $n \geq 3$ is projectively normal. 

\end{proposition}

\begin{proposition}[\cite{Hul}, Theorem IV.1.3]
\label{tw}
Every elliptic normal curve $C_n \subset \mathbb{P}^{n-1}$ of degree $n \geq 4$ is a scheme-theoretic intersection of the quadrics of rank 3 which contain it. 

\end{proposition}

\section{Elliptic normal curves and quadric hypersurfaces}

In this section we want to determine the space of quadric hypersurfaces through an elliptic normal curve of degree 6. 

\begin{theorem}
\label{propcz}

Let $C = \mathbb{C} / (\mathbb{Z}\omega_1 + \mathbb{Z}\omega_2)$ be an elliptic curve and $C_6 \subset \mathbb{P}^5$ be its embedding as a normal elliptic sextic. Then there exists a 9-dimensional space of quadric hypersurfaces containing $C_6$. A basis of this space is given by \begin{center}\begin{tabular}{ccl}\
$Q_0$  & $=$ & $x_0^2 + x_3^2 + \alpha(x_2x_4 + x_5x_1)$, \\
$Q_1$  & $=$ & $x_1^2 + x_4^2 + \alpha(x_3x_5 + x_0x_2)$,  \\
$Q_2$  & $=$ & $x_2^2 + x_5^2 + \alpha(x_4x_0 + x_1x_3)$, \\
$Q_0'$  & $=$ & $x_0^2 - x_3^2 + \beta(x_2x_4 - x_5x_1)$, \\
$Q_1'$  & $=$ & $x_1^2 - x_4^2 + \beta(x_3x_5 - x_0x_2)$,  \\
$Q_2'$  & $=$ & $x_2^2 - x_5^2 + \beta(x_4x_0 - x_1x_3)$, \\
$Q_0''$  & $=$ & $x_0x_1 + x_3x_4 + \gamma x_2x_5$, \\
$Q_1''$  & $=$ & $x_1x_2 + x_4x_5 + \gamma x_3x_0$,  \\
$Q_2''$  & $=$ & $x_2x_3 + x_5x_0 + \gamma x_4x_1$
\end{tabular}\end{center}

with \begin{center}\begin{tabular}{ccl}\
$\alpha$  & $=$ & $-\frac{x_3^2(\omega)}{x_2(\omega)x_4(\omega) + x_5(\omega)x_1(\omega)}$, \\
$\beta$  & $=$ & $\frac{x_3^2(\omega)}{x_2(\omega)x_4(\omega) - x_5(\omega)x_1(\omega)}$,  \\
$\gamma$  & $=$ & $-\frac{x_3(\omega)x_4(\omega)}{x_2(\omega)x_5(\omega)}$
\end{tabular}\end{center} 

where $\omega = \frac{\omega_1}{2} + \frac{\omega_2}{12}$.

\end{theorem}

\begin{proof}

Consider the following exact sequence: 
\[ \begin{tikzcd}
0 \arrow{r} & \mathcal{I}_{C_6}(2) \arrow{r} & \mathcal{O}_{\mathbb{P}^5}(2) \arrow{r} & \mathcal{O}_{C_6}(2) \arrow{r} & 0. 
\end{tikzcd}
\]
Since $C_6$ is projectively normal, we obtain $$h^0(\mathcal{I}_{C_6}(2)) = h^0(\mathcal{O}_{\mathbb{P}^5}(2)) - h^0(\mathcal{O}_{C_6}(2)) = 21 - 12 = 9.$$ 
In what follows, we often use two basic facts from representation theory which can be found in \cite[Proposition 1.5 and Proposition 1.8]{FH}. We leave related computations to the reader. 
Let $V^* = H^0(\mathcal{O}_{\mathbb{P}^5}(1))$. 
In order to find $H^0(\mathcal{I}_{C_6}(2))$ consider the $H_6$-module $$S^2V^* = H^0(\mathcal{O}_{\mathbb{P}^5}(2)),$$ 
which has the following decomposition  into irreducible $H_6$ modules: \begin{equation}\label{eq1}S^2V^* = V_0^+ \oplus V_0^-\oplus V_1^+ \oplus V_1^- \oplus V_2^+ \oplus V_2^- \oplus V_3,\end{equation} 
where 
\begin{center}\begin{tabular}{ccl}\
$V_0^{\pm}$  & = & $\langle x_0^2 \pm x_3^2, x_1^2 \pm x_4^2, x_2^2 \pm x_5^2 \rangle$, \\
$V_1^{\pm}$ & = & $\langle x_0x_1 \pm x_3x_4, x_1x_2 \pm x_4x_5, x_2x_3 \pm x_5x_0 \rangle$, \\
$V_2^{\pm}$ & = & $\langle x_0x_2 \pm x_3x_5, x_1x_3 \pm x_4x_0, x_2x_4 \pm x_5x_1 \rangle$, \\
$V_3$ & = & $\langle x_0x_3, x_1x_4, x_2x_5 \rangle$.
\end{tabular}\end{center} 

Let $V_{\tau}$ be the subspace of $S^2V^*$ consisting of elements invariant under $\tau$. Then obviously $$V_{\tau} = \langle x_0^2, x_3^2, x_2x_4, x_5x_1 \rangle.$$ 

Denote by $H_6(V_\tau)$ the smallest subspace of $S^2V^*$ containing $V_{\tau}$ and invariant under $H_6$. Then $$H_6(V_{\tau}) = V_0^{+} \oplus V_0^{-} \oplus V_2^{+} \oplus V_2^{-}.$$ 
Analogously, $$H_6(V_{\sigma^3}) = V_0^+ \oplus V_1^+ \oplus V_2^+ \oplus V_3.$$ 

From those two decompositions we can conclude that $V_0^+$ can be isomorphic (as an $H_6$-module) only to $V_2^+$ , $V_0^-$ can be isomorphic (as an $H_6$-module) only to $V_2^-$ and $V_1^+$ only to $V_3$. Those pairs of $H_6$-modules are indeed isomorphic, hence those are all isomorphisms between the modules of decomposition \eqref{eq1}.

We have decomposition of $H^0(\mathcal{I}_{C_6}(2))$ into irreducible $H_6$-modules $$H^0(\mathcal{I}_{C_6}(2)) = W_0 \oplus W_1 \oplus W_2$$ 
with $W_i$ invariant under $H_6$ and $\dim(W_i) = 3$. Note that the intersection of each $W_i$ with each $H_6-$module in the decomposition \eqref{eq1} is either trivial or the whole $W_i$. It cannot be $W_i$ in any of these cases because then there is a module in the decomposition \eqref{eq1} which is contained in $H^0(\mathcal{I}_{C_6}(2))$, which is impossible. Indeed, since $x_i(\frac12(\omega_1 + \frac{\omega_2}{6})) = 0$ if and only if $i = 0$ we can see that each module contains a generator which does not vanish on the whole $C_6$. It follows that none of $W_i$ is isomorphic to $V_1^-$ as $H_6$-modules. Indeed, otherwise such $W_i$ would be equal to $V_1^-$. 

Now, if any two of $W_i$'s (suppose that $W_0$ and $W_1$ without loss of generality) are isomorphic to the same $H_6$-module in decomposition \eqref{eq1} (suppose that $V_0^+$ without loss of generality) then, since there is only one subrepresentation of $S^2V^*$ isomorphic to $V_0^+ \oplus V_2^+$, we have $W_0 \oplus W_1 = V_0^+ \oplus V_2^+$ hence $V_0^+ \subset H^0(\mathcal{I}_{C_6}(2))$ - contradiction. Thus,
\begin{center}\begin{tabular}{ccl}\
$W_0$  & $\subset$ & $V_0^+ \oplus V_2^+$, \\
$W_1$  & $\subset$ & $V_0^- \oplus V_2^-$,  \\
$W_2$  & $\subset$ & $V_1^+ \oplus V_3$ 
\end{tabular}\end{center} 
up to a permutation of $W_i$'s. 

Since $W_0$ is isomorphic to $V_0^{+}$ and $V_2^+$ as an $H_6$-module we can choose a basis $Q_0,Q_1,Q_2$ of $W_0$ such that $Q_{i+1} = \sigma(Q_i)$ and $\tau(Q_0) = Q_0$. Hence \begin{center}\begin{tabular}{ccl}\
$Q_0$  & $=$ & $x_0^2 + x_3^2 + \alpha(x_2x_4 + x_5x_1)$, \\
$Q_1$  & $=$ & $x_1^2 + x_4^2 + \alpha(x_3x_5 + x_0x_2)$,  \\
$Q_2$  & $=$ & $x_2^2 + x_5^2 + \alpha(x_4x_0 + x_1x_3)$
\end{tabular}\end{center} 

with $$\alpha = -\frac{x_3^2(\omega)}{x_2(\omega)x_4(\omega) + x_5(\omega)x_1(\omega)}, \:\:\:\: \text{where } \omega = \frac{\omega_1}{2} + \frac{\omega_2}{12}.$$

Analogously, the basis of $W_1$ is \begin{center}\begin{tabular}{ccl}\
$Q_0'$  & $=$ & $x_0^2 - x_3^2 + \beta(x_2x_4 - x_5x_1)$, \\
$Q_1'$  & $=$ & $x_1^2 - x_4^2 + \beta(x_3x_5 - x_0x_2)$,  \\
$Q_2'$  & $=$ & $x_2^2 - x_5^2 + \beta(x_4x_0 - x_1x_3)$
\end{tabular}\end{center} 

with $$\beta= \frac{x_3^2(\omega)}{x_2(\omega)x_4(\omega) - x_5(\omega)x_1(\omega)}, \:\:\:\: \text{where } \omega = \frac{\omega_1}{2} + \frac{\omega_2}{12}.$$

Since $W_2$ is isomorphic to $V_1^{+}$ and $V_3$ as an $H_6$-module we can choose a basis $Q_0'',Q_1'',Q_2''$ of $W_2$ such that $Q_{i+1}'' = \sigma(Q_i'')$ and $\tau(Q_0'') = \varepsilon \cdot Q_0''$. Hence \begin{center}\begin{tabular}{ccl}\
$Q_0''$  & $=$ & $x_0x_1 + x_3x_4 + \gamma x_2x_5$, \\
$Q_1''$  & $=$ & $x_1x_2 + x_4x_5 + \gamma x_3x_0$,  \\
$Q_2''$  & $=$ & $x_2x_3 + x_5x_0 + \gamma x_4x_1$
\end{tabular}\end{center} 

with $$\gamma = -\frac{x_3(\omega)x_4(\omega)}{x_2(\omega)x_5(\omega)}, \:\:\:\: \text{where } \omega = \frac{\omega_1}{2} + \frac{\omega_2}{12}.$$ 

Thus, we have determined the basis of $H^0(\mathcal{I}_{C_6}(2))$.

\end{proof}

\begin{remark}

In \cite[Proposition 7.7 and Section 9]{KK} Kaneko and Kuwata use different method based on the properties of theta-functions and the classical Jacobi's identity and obtain similar results. If we denote the space $H^0(\mathcal{I}_{C_6}(2)) \cap \{x_ix_j \: | \: i + j \equiv k \text{ mod } 6\}$ by $V_k$ one can see that both $$a_1^2X_0^2 + a_2^2X_3^2 - a_0^2X_1X_5 - a_3^2X_2X_4,$$ $$a_2^2X_0^2 + a_1^2X_3^2 - a_3^2X_1X_5 - a_0^2X_2X_4$$ found by them and our $Q_0, Q_0'$ from Theorem \ref{propcz} form two bases of the space $V_0$ and similarly for $V_k$ with $0 < k < 6$.

\end{remark}

\begin{lemma}
Let $\alpha, \beta$ and $\gamma$ denote the constants defined in Theorem \ref{propcz}. Then the following relations hold: $$\alpha\beta(\alpha + \beta) = -2,$$ $$\gamma = \alpha\beta.$$ 
\end{lemma}

\begin{proof}

    Let $\omega = \frac{\omega_1}{2} + \frac{\omega_2}{12}$ be as in the previous theorem. Note that $\alpha, \beta$ and $\gamma$ are non-zero because $x_i(\omega) = 0 \Leftrightarrow i = 0$.  Using the equations defining $Q_0$ and $Q_0'$ we obtain 
    \begin{align*} 
    \frac{1}{\alpha} &= -\frac{x_2(\omega)x_4(\omega) + x_5(\omega)x_1(\omega)}{x_3^2(\omega)}, \\ 
    \frac{1}{\beta} &= \frac{x_2(\omega)x_4(\omega) - x_5(\omega)x_1(\omega)}{x_3^2(\omega)}.
    \end{align*}
    
    Hence $$\frac{1}{\alpha} + \frac{1}{\beta} =- \frac{2x_5(\omega)x_1(\omega)}{x_3^2(\omega)}.$$
    
    and, in particular, $\alpha \neq - \beta$.
    
    Now, using the equations defining $Q_1$ and $Q_1'$ we get 
    $$\alpha + \beta = -\frac{2x_1^2(\omega)}{x_3(\omega)x_5(\omega)}$$
    
    Analogously, from the equations of $Q_2$ and $Q_2'$ \begin{equation}\label{eqq}\alpha + \beta = -\frac{2x_5^2(\omega)}{x_1(\omega)x_3(\omega)}\end{equation}
    
    Multiplying the equations above we obtain $$(\alpha + \beta)^2 = \frac{4x_1(\omega)x_5(\omega)}{x_3^2(\omega)} = -2\left(\frac{1}{\alpha} + \frac{1}{\beta}\right).$$ 
    
    We can multiply both sides by $\frac{\alpha\beta}{\alpha + \beta}$ to finally get \begin{equation}\label{eqq2}(\alpha + \beta)\alpha\beta = -2.\end{equation}
    
    From equations (\ref{eqq}) and (\ref{eqq2}) we get
    $$\alpha\beta = \frac{x_1(\omega)x_3(\omega)}{x_5^2(\omega)}.$$
    
    Now, using the equations defining $Q_0''$ and $Q_1''$ we get
    \begin{align*} 
    \gamma = -\frac{x_3(\omega)x_4(\omega)}{x_2(\omega)x_5(\omega)}, \:\:\:\:\:
    x_2(\omega) = -\frac{x_4(\omega)x_5(\omega)}{x_1(\omega)}.
    \end{align*}
    
    It follows that $$\gamma = \frac{x_1(\omega)x_3(\omega)}{x_5^2(\omega)} = \alpha\beta.$$

\end{proof}

Now let us find the generators of $I(\Sec{C_6})$ using the basis we found in Proposition \ref{tw}. Let $C_n \subset \mathbb{P}^{n - 1}$ be an elliptic normal curve. Let us start with two lemmas.   

\begin{lemma}[\cite{F}, Lemma 3.2] 
\label{lemds}

If $n \geq 6$ then $I(Sec(C_n))$ is generated by cubics. A cubic form $f$ vanishes on $\Sec(C_n)$ if and only if it is singular at every point on $C_n$, equivalently $\frac{\partial f}{\partial x_i} \in I(C_n)$ for all $0 \leq i \leq n-1$.

\end{lemma}

\begin{lemma}[\cite{F}, Lemma 3.4] 

The space of cubics vanishing on $\Sec(C_n)$ has dimension at most $\frac{n(n - 4)(n - 5)}{6}$. 

\end{lemma}

We can see that $I(Sec(C_6))$ is generated by at most two cubics. Moreover, we can find explicit equations of those cubics using the second part of Lemma \ref{lemds} and Theorem \ref{propcz}. 

\begin{theorem}

Recall that $\sigma(x_n) = x_{n - 1}$ for all $n$ and $\alpha, \beta$ and $\gamma$ are defined in Theorem \ref{propcz}. Then the ideal $I(\Sec(C_6))$ is generated by two cubic surfaces $F_1$ and $F_2$ given by $$F_1 = 2(\alpha^2\beta^2 - \alpha - \beta)x_0x_2x_4 + \sum_{i = 0}^2\sigma^i(-2x_0^3 + 2(\beta - \alpha)x_1x_2x_3 + \alpha\beta(\beta - \alpha)x_0x_3^2),$$
$$F_2 = 2(\alpha^2\beta^2 - \alpha - \beta)x_1x_3x_5 + \sum_{i = 0}^2\sigma^i(-2x_1^3 + 2(\beta - \alpha)x_2x_3x_4 + \alpha\beta(\beta - \alpha)x_1x_4^2).$$

\end{theorem}

\begin{proof}

Note that all monomials in $F_1$ and $F_2$ are distinct and $F_1 = \sigma(F_2)$ and $F_2 = \sigma(F_1)$. Therefore, they are linearly independent and it is enough to check that 
$\frac{\partial F_1}{\partial x_0},\frac{\partial F_1}{\partial x_1} \in I(C_6)$. One can easily check that $$\frac{\partial F_1}{\partial x_0} = \left(-3 + \frac{\alpha\beta(\beta - \alpha)}{2}\right)\cdot Q_0 + \left(-3 - \frac{\alpha\beta(\beta - \alpha)}{2}\right)\cdot Q_0' \in I(C_6)$$ and $$\frac{\partial F_1}{\partial x_1} = 2(\beta - \alpha)\cdot Q_2'' \in I(C_6),$$ where $Q_0, Q_0'$ and $Q_2''$ are quadrics defined in Theorem \ref{propcz}.    

\end{proof}

\begin{remark}

We used a computer algebra system \textit{Singular} (\cite{Sing}) to compute $F_1$ and $F_2$. 

\end{remark}

\section{Geometry of an elliptic normal curve $C_6 \subset \mathbb{P}^5$}

In this section we will prove the main results of the paper. 

We will denote the image of $C_6$ under the projection from point $P \notin \Sec(C_6)$ by $C_p$. The curve $C_p$ is an elliptic curve of degree 6 in $\mathbb{P}^4$.

\begin{proposition}
\label{prop3}

The curve $C_p$ is $k$-normal for all $k \geq 2$ and
$$h^0(\mathcal{I}_{C_p}(2)) = 3,$$ $$h^0(\mathcal{I}_{C_p}(3)) = 17.$$

\end{proposition}

\begin{proof}

Let $\pi$ be a projection map from $P$. Then the following diagram commutes and the rows are short exact sequences: 

\[ \begin{tikzcd}
0 \arrow{r} & \mathcal{I}_{C_6}(k) \arrow{r} & \mathcal{O}_{\mathbb{P}^5}(k) \arrow{r} & \mathcal{O}_{C_6}(k) \arrow{r} & 0 \\%
0 \arrow{r} & \mathcal{I}_{C_p}(k) \arrow{r} \arrow[swap]{u}{\pi^*} & \mathcal{O}_{\mathbb{P}^4}(k) \arrow{r} \arrow[swap]{u}{\pi^*} & \mathcal{O}_{C_p}(k) \arrow{r} \arrow[swap]{u}{\pi^*} & 0
\end{tikzcd}
\]
Since $C_6$ is projectively normal we obtain a commutative diagram with rows being exact
\[ \begin{tikzcd}
0 \arrow{r} & H^0(\mathcal{I}_{C_6}(k)) \arrow{r} & H^0(\mathcal{O}_{\mathbb{P}^5}(k)) \arrow{r}{\varphi} & H^0(\mathcal{O}_{C_6}(k)) \arrow{r} & 0 \\%
0 \arrow{r} & H^0(\mathcal{I}_{C_p}(k)) \arrow{r} \arrow[swap]{u}{\pi_1^*} & H^0(\mathcal{O}_{\mathbb{P}^4}(k)) \arrow{r}{\psi} \arrow[swap]{u}{\pi_2^*} & H^0(\mathcal{O}_{C_p}(k)) \arrow[swap]{u}{\pi_3^*} 
\end{tikzcd}
\]

It is enough to show that $\psi$ is surjective. Note that $\pi_i^*$'s (maps induced by $\pi^*$) are injective. From $h^0(\mathcal{O}_{C_6}(k)) = h^0(\mathcal{O}_{C_p}(k)) = 6k$ it follows that $\pi_3^*$ is an isomorphism. Thus, it is enough to show that $\varphi \circ \pi_2^*$ is surjective. Take any $\langle v \rangle \subset H^0(\mathcal{O}_{C_6}(k))$. It has codimension $6k - 1$ in $H^0(\mathcal{O}_{C_6}(k))$ and, since $\varphi$ is a surjection, codimension of $\varphi^{-1}(\langle v \rangle)$ in $H^0(\mathcal{O}_{\mathbb{P}^5}(k))$ is also equal to $6k - 1$ hence $$\dim \varphi^{-1}(\langle v \rangle) = \binom{5 + k}{5} - (6k - 1).$$ 
But then $$\dim \pi_2^*(H^0(\mathcal{O}_{\mathbb{P}^4}(k))) + \dim \varphi^{-1}(\langle v \rangle) = \binom{4 + k}{4} + \binom{5 + k}{5} - 6k + 1,$$ which is greater than $h^0(\mathcal{O}_{\mathbb{P}^5}(k)) = \binom{5 + k}{5}$ for all $k \geq 2$ so $$\pi_2^*(H^0(\mathcal{O}_{\mathbb{P}^4}(k))) \cap \varphi^{-1}(\langle v \rangle) \neq \{0\}.$$ 
It follows that $\varphi \circ \pi_2^*$ is indeed surjective.

From the exact sequence 
\[ \begin{tikzcd}
0 \arrow{r} & \mathcal{I}_{C_{p}}(2) \arrow{r} & \mathcal{O}_{\mathbb{P}^4}(2) \arrow{r} & \mathcal{O}_{C_{p}}(2) \arrow{r} & 0 
\end{tikzcd}
\]
we obtain $$h^0(\mathcal{I}_{C_p}(2)) = h^0(\mathcal{O}_{\mathbb{P}^4}(2)) - h^0(\mathcal{O}_{C_p}(2)) = 15 - 12 = 3$$ and similarly $$h^0(\mathcal{I}_{C_p}(3)) = h^0(\mathcal{O}_{\mathbb{P}^4}(3)) - h^0(\mathcal{O}_{C_p}(3)) = 35 - 18 = 17.$$

\end{proof}

We can go further and consider an image of $C_p$ under a projection from a point $Q \in \mathbb{P}^4 \setminus \Sec^2(C_p)$. We will denote this image by $C_{pq}$. The proof of the following result is analogous. 

\begin{proposition}
\label{prprpr}

The curve $C_{pq}$ is $k$-normal for all $k \geq 3$ and $$h^0(\mathcal{I}_{C_{pq}}(3)) = 2,$$ $$h^0(\mathcal{I}_{C_{pq}}(4)) = 11.$$

\end{proposition}

Our next aim is to show that ideals of $C_p$ and $C_{pq}$ are generated in degree 3 and 4 respectively. We will use the theory of regularity of coherent sheaves to do this. 

\begin{definition}

Let $\mathcal{F}$ be a coherent sheaf on $\mathbb{P}^n$. We say that $\mathcal{F}$ is $m$-regular if $$H^i(\mathbb{P}^n, \mathcal{F}(m - i)) = 0 \:\:\:\: \mbox {for all } i > 0.$$ 

\end{definition}

\begin{proposition}[\cite{Mum}, Lecture 14]
\label{mum}
Let $\mathcal{F}$ be an $m$-regular coherent sheaf on $\mathbb{P}^n$. Then $H^0(\mathbb{P}^n, \mathcal{F}(k))$ is spanned by $$H^0(\mathbb{P}^n, \mathcal{F}(k - 1)) \otimes H^0(\mathbb{P}^n, \mathcal{O}_{\mathbb{P}^n}(1)) \:\: \mbox{ if } k > m.$$

\end{proposition}

This proposition applied to $\mathcal{I}_X$ of some projective variety $X \subset \mathbb{P}^n$ says that if $\mathcal{I}_X$ is $m$-regular, then $I(X)$ is generated in degree $m$. 

\begin{lemma}
\label{tra}
Let $C \subset \mathbb{P}^n$ be a curve and $m > 0$ be a positive integer. Then $\mathcal{I}_C$ is $m$-regular if and only if $C$ is $(m - 1)$-normal and $H^1(\mathcal{O}_C(m - 2)) = 0$. 

\end{lemma}

\begin{proof}

Recall that for any $k$ we have an exact sequence 

\[ \begin{tikzcd}
0 \arrow{r} & \mathcal{I}_{C}(k) \arrow{r} & \mathcal{O}_{\mathbb{P}^n}(k) \arrow{r} & \mathcal{O}_{C}(k) \arrow{r} & 0. 
\end{tikzcd}
\]
If $\mathcal{I}_C$ is $m$-regular, then $H^1(\mathcal{I}_C(m-1)) = 0$, hence we have a short exact sequence 
\[ \begin{tikzcd}
H^0(\mathcal{O}_{\mathbb{P}^n}(m - 1)) \arrow{r} & H^0(\mathcal{O}_{C}(m - 1)) \arrow{r} & 0. 
\end{tikzcd}
\]
Therefore, $C$ is $(m - 1)$-normal. From the following fragment of the long exact sequence of cohomologies 

\[ \begin{tikzcd}
H^1(\mathcal{O}_{\mathbb{P}^n}(m - 2)) \arrow{r} & H^1(\mathcal{O}_{C}(m - 2)) \arrow{r} & H^2(\mathcal{I}_{C}(m - 2))  
\end{tikzcd}
\]
and the fact that both $H^1(\mathcal{O}_{\mathbb{P}^n}(m - 2))$ (see \cite[III proposition 2.1.12]{Grot}) and $H^2(\mathcal{I}_{C}(m - 2))$ (since $\mathcal{I}_C$ is $m$-regular) are equal to zero, we get $H^1(\mathcal{O}_{C}(m - 2)) = 0$. 

For the proof of the opposite direction, consider the following fragment of the long exact sequence of cohomologies for $q > 1$

\[ \begin{tikzcd}
H^{q - 1}(\mathcal{O}_{C}(m - q)) \arrow{r} & H^q(\mathcal{I}_{C}(m - q)) \arrow{r} & H^q(\mathcal{O}_{\mathbb{P}^n}(m - q))  
\end{tikzcd}
\]
Since $C$ is a one-dimensional projective variety, we have $H^{q - 1}(\mathcal{O}_{C}(m - q)) = 0$ for $q > 2$ and $H^{1}(\mathcal{O}_{C}(m - 2)) = 0$ from the assumptions. Also, by \cite[III proposition 2.1.12]{Grot}, $H^q(\mathcal{O}_{\mathbb{P}^n}(m - q)) = 0$ for $q \neq n$. Note that $H^n(\mathcal{O}_{\mathbb{P}^n}(m - n))$ is also equal to 0 because $m - n \geq -n$. Thus, we always have zeros on the left and on the right in the exact sequence. It follows that $H^q(\mathcal{I}_{C}(m - q)) = 0$.

For $q = 1$ we have an exact sequence 

\[ \begin{tikzcd}
H^0(\mathcal{O}_{\mathbb{P}^n}(m - 1)) \arrow{r} & H^{0}(\mathcal{O}_{C}(m - 1)) \arrow{r} & H^1(\mathcal{I}_{C}(m - 1)) \arrow{r} & 0 
\end{tikzcd}
\]
Since $C$ is $(m - 1)$-normal, $H^0(\mathcal{O}_{\mathbb{P}^n}(m - 1)) \to H^{0}(\mathcal{O}_{C}(m - 1))$
is surjective, hence $H^1(\mathcal{I}_{C}(m - 1)) = 0$. 
\end{proof}

\begin{proposition} 
\label{k}
The following hold:

\begin{enumerate}
    \item The ideal of the curve $C_p$ is generated in degree 3.
    \item The ideal of the curve $C_{pq}$ is generated in degree 4.
\end{enumerate}

\end{proposition}

\begin{proof}

Note that for any elliptic curve $C \subset \mathbb{P}^n$ we have $H^1(\mathcal{O}_C(t)) = 0$ for any $t > 0$. Indeed, by Serre duality $H^1(\mathcal{O}_C(t)) = H^0(\mathcal{O}_C(-t)) = 0$.

\begin{enumerate}

    \item Since $C_p$ is $2$-normal and $H^1(\mathcal{O}_{C_p}(1)) = 0$, we get that $\mathcal{I}_{C_p}$ is $3$-regular by Lemma \ref{tra}.
    \item Analogously, $C_{pq}$ is $3$-normal and $H^1(\mathcal{O}_{C_{pq}}(2)) = 0$, hence $\mathcal{I}_{C_{pq}}$ is $4$-regular.
\end{enumerate}

Therefore, by Proposition \ref{mum}, ideals of $C_p$ and $C_{pq}$ are generated in degree 3 and 4 respectively. 

\end{proof}

Our next aim is to show that if $P$ is a general point of $\mathbb{P}^5$ then the ideal $I(C_p)$ of the curve $C_p$ is generated by three polynomials of degree 2 and two polynomials of degree 3. 

It follows from (\cite[Lemma 1 and main Theorem]{Lan}) that $\dim \Sec^3(C_6) = 5$ hence $\Sec^3(C_6) = \mathbb{P}^5$.

\begin{definition}
\label{defka}
Let $X \subset \mathbb{P}^r$ be a closed, irreducible subvariety of dimension $n$, not lying on a hyperplane and such that $\dim \Sec^k(X) = k (n + 1) - 1$. Then through a general point of $\Sec^k(X)$ there passes a finite number of $k$-secant $(k-1)$-planes. In this case, let $\sec_k(X)$ denote the number of such $(k-1)$-planes.

\end{definition}

For the rest of this chapter we impose two conditions on $P$:

\begin{enumerate}
    \item The point $P$ is a general point of $\Sec^3(C_6) = \mathbb{P}^5$ in the sense of Definition \ref{defka}, \\i. e. the number of $3$-secant $2$-planes passing through $P$ is exactly $\sec_k(C_6)$.
    \item The point $P$ is not a vertex of a quadric hypersurface of rank 3 in $\mathbb{P}^5$ containing $C_6$.
\end{enumerate}

\begin{remark}

Although the set of points $P \in \mathbb{P}^5$, satisfying the above conditions contains an open subset of $\mathbb{P}^5$ (we will see the proof of this fact in Lemma \ref{lemmaa}), it is important to mention that the conditions can be weakened:  

The first condition can be weakened to: "the number of $3$-secant $2$-planes passing through $P$ is greater than 1". 

The second condition is essential for the proof of Lemma \ref{lemmaa} and, perhaps, can also be weakened. In this case the proof of Lemma \ref{lemmaa} must be modified, and we leave this direction for further research. 

\end{remark}

Since $C_6$ is not a rational normal curve we have $\sec_3(C_6) > 1$ (see \cite[Theorem 3.4]{CJ}). From our assumptions there exist at least two planes spanned by triples of points on $C_6$ containing $P$. Note that such triples are disjoint. Indeed, if there are $m \in \{1,2\}$ common points then the intersection of planes contains an $m$-plane spanned by those $m$ points and $P$. It follows that those $6 - m$ points span a linear subspace of dimension $4 - m$ which contradicts Lemma \ref{lemka}. We will denote these points by $(R_1,R_2,R_3)$ and $(T_1,T_2,T_3)$. Let $S_1 = \spann(R_1, R_2, R_3)$, $S_2 = \spann(T_1, T_2, T_3)$ and $H = \spann(R_1, R_2, R_3, T_1, T_2, T_3)$. Since $P \in S_1 \cap S_2$ and span$(R_1, R_2, R_3, T_1, T_2) \simeq \mathbb{P}^4$, we have $H = $ span$(R_1, R_2, R_3, T_1, T_2, T_3) \simeq \mathbb{P}^4$. 

Let $f \in H^0(\mathcal{I}_{C_p}(2))$ be a homogeneous polynomial of degree two vanishing on $C_p$ and $g = \pi_1^*f \in H^0(\mathcal{I}_C(2))$ be its pullback. Then $g$ vanishes on the lines $\overline{PR_1}$, $\overline{PR_2}$, $\overline{PR_3}$. Since $S_1$ has degree 1 and $V(g)$ is a hypersurface of degree $2$ by Bézout's theorem their intersection is either the whole $H$ or a curve of degree 2. But the intersection contains three lines so it cannot be a curve of degree two, hence $S_1 \subset V(g)$. Thus, the line $\pi(S_1)$ is contained in $V(f)$. Similarly, the line through $\pi(T_1), \pi(T_2)$ and $\pi(T_3)$ is contained in $V(f)$. 

Let $q_1, q_2$ and $q_3$ be a basis of $H^0(\mathcal{I}_{C_p}(2))$ and $Q_i = V(q_i)$. Our aim is to show that $Q_1 \cap Q_2 \cap Q_3$ is a complete intersection, i.e., it is a curve. But firstly, we need an additional result. Let us start with the definition.

\begin{definition}[\cite{CC}]

Let $X\subset \mathbb{P}^r$ be a projective variety. An abstract $k$-th secant variety $\Sec_{X}^k$ of $X$, $\Sec_{X}^k \subseteq \Symm^{k}(X)\times \mathbb{P}^r$ is the Zariski closure of the set of all pairs $([P_0,\ldots,P_{k -1 }],x)$ such that $P_0, \ldots, P_{k - 1} \in X$ are linearly independent, smooth points and $x \in \spann(P_0,\ldots, P_{k - 1})$. 

One has the surjective map $p_X^k: \Sec_X^k \to \Sec^k(X) \subseteq \mathbb{P}^r$, that is, the projection to the second factor, which we call the $k$-secant map of $X$.  

\end{definition}

\begin{lemma}
\label{lemmaa}
There exist 4 points $A_1, A_2, A_3, A_4 \in C_6$ different from $R_i, T_i$ defined above, such that $\spann(A_1, A_2, A_3, A_4)$ contains $P$.

\end{lemma}

\begin{proof}

Note that the general fibre of the $4$-secant map $p^4_{C_6}$ has dimension 2. Let us show that the set $$M_{R_1} = \{([R_1,P_0, P_1, P_2], P)\} \subset \Sec_{C_6}^4,$$ where $P_0, P_1, P_2$ are arbitrary points of $C_6$ such that $P \in \spann(R_1, P_0, P_1, P_2)$, has dimension 1. It is enough to show that for a generic point $A \in C_6$ there exist at most finitely many pairs $(B_1^i, B_2^i) \in C_6 \times C_6$ such that $P \in \spann(R_1, A, B_1^i, B_2^i)$. In fact, we can show this for any $A \in C_6 \setminus \{R_1, R_2, R_3\}$. So let $A$ be such point and $\{(B_1^i,B_2^i)\}_{i \in I}$ be a set of all pairs of points of $C_6$ such that $P \in \spann(R_1, A, B_1^i, B_2^i)$ for all $i \in I$. Let $C_6' \subset \mathbb{P}^3$ be the image of $C_6$ by the composition $\pi_{\pi_{R_1}(A)} \circ \pi_{R_1}$ of projections from $R_1$ and then from $\pi_{R_1}(A)$. Also, let $E_j^i = \pi_{\pi_{R_1}(A)} \circ \pi_{R_1} (B_j^i)$, $R_2' = \pi_{\pi_{R_1}(A)} \circ \pi_{R_1} (R_2)$ and $R_3' = \pi_{\pi_{R_1}(A)} \circ \pi_{R_1} (R_3)$. Then $P' = \pi_{\pi_{R_1}(A)} \circ \pi_{R_1} (P) \in \overline{E_1^iE_2^i}$ for all $i \in I$. Let $E \in C_6'$ be any point other than $R_2', R_3'$. Then $R_2', R_3'$ and $E$ span a projective plane which intersects $C_6'$ at some point $E'$ (since $C_6'$ has degree 4). Since the lines $\overline{EE'}$ and $\overline{R_2'R_3'}$ span a projective plane and not $\mathbb{P}^3$, they intersect at some point $J$. Thus, we have a function $\varphi$ from an open subset of $C_6'$ to $\overline{R_2'R_3'} \simeq \mathbb{P}^1$ such that $\varphi(E) = J$. Note that this function is a rational map. Thus, the image of $\varphi$ is either a point, so $\{P'\}$, or an open subset of $\mathbb{P}^1$. 

If $\imm \varphi = \{P'\}$ then the restriction of projection map from $P'$, $\pi_{P'|C_6'}: C_6' \to \pi_{P'}(C_6')$ is 2 to 1 so $\pi_{P'}(C_6')$ is a curve of degree 2 in $\mathbb{P}^2$ hence it is a conic of rank 3 (it can be written as $XY - Z^2$ after some change of coordinates). Thus, $P'$ is a vertex of a quadric surface of rank 3 in $\mathbb{P}^3$. From \cite[page 28]{Hul} we know that there are four singular quadric surfaces containing $C_6'$ (since $C_6'$ is a normal elliptic curve of degree 4 in $\mathbb{P}^3)$. Thus, $P'$ is one of four vertices of these quadric cones. The set of points $$K = \bigcup_{A \in C_6}\pi_{\pi_{R_1}(A)}^{-1}(P')$$ for a fixed $R_1$ is a constructible set of dimension 2 since for any $A$ the set $\pi_{\pi_{R_1}(A)}^{-1}(P')$ is a line, hence a one-dimensional set, and we choose $A$ from $C_6$, which is one-dimensional as well. Analogously, the set of points $$K' = \bigcup_{R_1 \in C_6}\pi^{-1}_{R_1}(K)$$ is a 4-dimensional constructible set. But we have chosen $P$ to be a general point which, in particular, does not belong to $K'$. Thus, $\imm \varphi$ is not a point. 

On the other hand, if $\imm \varphi$ is an open subset of $\mathbb{P}^1$ then, since both $C_6'$ and $\mathbb{P}^1$ are one-dimensional, the fibers of $\varphi$, and, in particular, the fiber over $P'$, are finite. Therefore, $I$ is a finite set, hence the set $M_{R_1}$ is indeed one-dimensional. 

The same result holds for points $R_2, R_3, T_1, T_2, T_3$ and, since the fibre over $P$ is at least two-dimensional, we can conclude, that there exist some $A_1, A_2, A_3, A_4$ different from $R_i$ and $T_i$ and such that $([A_1, A_2, A_3, A_4], P) \subset \Sec_{C_6}^4$. 

\end{proof}

\begin{proposition}
\label{p1}
The intersection of quadric hypersurfaces $Q_1$, $Q_2$ and $Q_3$ is a curve.

\end{proposition}

\begin{proof}

Let $V_p \subset V^*$ be the space of linear forms vanishing at $P$.
By Lemma \ref{lemmaa} we can choose 4 points $A_1, A_2, A_3, A_4$ on $C_6$ such that each $A_i$ does not belong to $\{R_1, R_2, R_3, T_1, T_2, T_3\}$ and $P \in \spann(A_1, A_2, A_3, A_4)$. Let $h_0 \in V_p$ be such that $H = V(h_0)$. Let $H_1^1$ be a hyperplane spanned by $R_1, R_2, R_3$, $A_1$ and $A_2$. Then, since $C_6$ has degree 6, there exists $P_1^1 \in C_6$ such that $R_1 + R_2 + R_3 + A_1 + A_2 + P_1^1 \sim \mathcal{O}_{C_6}(H)$. Using the same argument, we can define 12 hyperplanes: 

\begin{enumerate}
    \item $H_1^1 = \spann( R_1, R_2, R_3, A_1, A_2, P_1^1)$; \:\:\:\:\:\:\:\:\:\:\:\:\:\: $H_2^1 = \spann( T_1, T_2, T_3, A_3, A_4, P_2^1)$.
    \item $H_1^2 = \spann( R_1, R_2, R_3, A_1, A_3, P_1^2)$; \:\:\:\:\:\:\:\:\:\:\:\:\:\: $H_2^2 = \spann( T_1, T_2, T_3, A_2, A_4, P_2^2)$.
    \item $H_1^3 = \spann( R_1, R_2, R_3, A_1, A_4, P_1^3)$; \:\:\:\:\:\:\:\:\:\:\:\:\:\: $H_2^3 = \spann( T_1, T_2, T_3, A_2, A_3, P_2^3)$.
    \item $H_1^4 = \spann( R_1, R_2, R_3, A_2, A_3, P_1^4)$; \:\:\:\:\:\:\:\:\:\:\:\:\:\: $H_2^4 = \spann( T_1, T_2, T_3, A_1, A_4, P_2^4)$.
    \item $H_1^5 = \spann( R_1, R_2, R_3, A_2, A_4, P_1^5)$; \:\:\:\:\:\:\:\:\:\:\:\:\:\: $H_2^5 = \spann( T_1, T_2, T_3, A_1, A_3, P_2^5)$.
    \item $H_1^6 = \spann( R_1, R_2, R_3, A_3, A_4, P_1^6)$; \:\:\:\:\:\:\:\:\:\:\:\:\:\: $H_2^6 = \spann( T_1, T_2, T_3, A_1, A_2, P_2^6)$.
\end{enumerate}

Now, let us show that all 12 hyperplanes are different. Indeed, the hyperplanes from the first column are clearly different from the ones in the second by Bézout's theorem. Without loss of generality suppose that $H_1^1 = H_1^m$ for some $m$. Then $$R_1 + R_2 + R_3 + A_1 + A_2 + A_j \sim \mathcal{O}_{C_6}(H)$$ for $j \in \{3,4\}$. But then planes $S_1$ and $\spann(A_1, A_2, A_3)$ intersect at some point $Q \neq P$. Thus, the line $\overline{PQ}$ is contained in $S_1 \cap \spann(A_1, A_2, A_3, A_4)$ hence $\{R_1, R_2, R_3, A_1, A_2, A_3, A_4\}$ lie both on $C_6$ and a hyperplane, which contradicts Bézout's theorem.    

Note that we have 
\begin{equation*}
\begin{split}
A_1 + A_2 + A_3 + A_4 + P_1^1 + P_2^1 &\sim (R_1 + R_2 + R_3 + A_1 + A_2 + P_1^1) \\ 
&+ (T_1 + T_2 + T_3 + A_3 + A_4 + P_2^1) \\
& - (R_1 + R_2 + R_3 + T_1 + T_2 + T_3) \sim  \mathcal{O}_{C_6}(H)    
\end{split}
\end{equation*}
so there is a hyperplane $H_3^1$ intersecting $C_6$ at $A_1, A_2, A_3, A_4, P_1^1$ and $P_2^1$. Analogously, there is a hyperplane $H_3^k$ intersecting $C_6$ in $A_1, A_2, A_3, A_4, P_1^k, P_2^k$ for $2 \leq k \leq 6$. Let $h_i^j \in V_p$ be a linear form such that $H_i^j = V(h_i^j)$. Then $f_j = h_0h_3^j - h_1^jh_2^j \in \Symm^2V_p$ and $V(f_j) \cap C_6$ contains $Z_j = \{R_1, R_2, R_3, T_1, T_2, T_3, A_1, A_2, A_3, A_4, P_1^j, P_2^j\}$ and we can multiply $h_3^j$ by some constant so that $V(f_j)$ contains one more point of $C_6$ not from $Z_j$. It follows from Bézout's theorem that $C_6 \subset V(f_j)$ for $1 \leq j \leq 6$.

Note that $N = Q_1 \cap Q_2$ has dimension 2 since $Q_1 \neq Q_2$ and $C_p$ is not contained in any hyperplane. Moreover, $N$ has degree 4 by Bézout's theorem. If $\dim(N \cap Q_3) = 2$ then its degree is less than 4 so it is 3 (\cite[Proposition 0]{EH}) - suppose this is the case. Let $\pi_1^*$ be the pull-back map from $H^0(\mathcal{I}_{C_p}(2))$ to $H^0(\mathcal{I}_{C_6}(2))$ induced from $\pi$. Then $$S = V(\pi_1^*q_1) \cap V(\pi_1^*q_2) \cap V(\pi_1^*q_3) \subset \mathbb{P}^5$$ is a projective variety of dimension 3 and degree 3. 
We know that $H \cap S$ contains $S_1$ and $S_2$ so, by Bézout's theorem, $H \cap S$ is either a union of three planes $S_1, S_2$ and $S_3$ or two planes, but the intersection is not transversal along one of them. 

Let us firstly consider the former case. It is enough to show that $S_3$ cannot be a subset of $V(q')$ for some $q' \in \Symm^2 V_p \cap H^0(\mathcal{I}_{C_6}(2)) = \spann(\pi_1^*q_1, \pi_1^*q_2, \pi_1^*q_3)$.  Since $S_3 \subset H$ and $S_3 \subset V(f_j)$ for each $j$ by our assumption, $S_3$ is contained in $V(h_1^jh_2^j)$. But $h_1^j$ and
$h_2^j$ are linear forms and $S_3$ is a linear space so $S_3$ is contained in one of $H_1^j$ and $H_2^j$ for each $j$. 

Firstly, let us show that $S_3$ cannot be a subset of more than 3 of $H_1^j$ (the same for $H_2^j$). Suppose that $S_3$ is a subset of 4 out of $H_1^j$'s. There are two possible cases, namely when the remaining two of $H_1^j$ share one of $A_i$'s and when they do not. Without loss of generality, we can assume that those four $H^1_j$'s are 

\begin{enumerate}
    \item $H_1^1$, $H_1^2$, $H_1^3$, $H_1^4$: since $S_1 \neq S_3$ we have $\dim \spann(S_1, S_3) \geq 3$. But since $S_1 \cup S_3 \subset H_1^j$ for $1 \leq j \leq 4$, we get $\spann(R_1, R_2, R_3, A_1) = H_1^1 \cap H_1^3 = \spann(S_1, S_3) = H_1^2 \cap H_1^4 = \spann(R_1, R_2, R_3, A_3)$ which contradicts lemma \ref{lemka} (5 points of $C_6$ on a linear space of dimension 3). 
    \item $H_1^1$, $H_1^2$, $H_1^5$, $H_1^6$: we can consider $H_1^1 \cap H_1^2$ and $H_1^5 \cap H_1^6$ and repeat the reasoning of item 1. 
\end{enumerate}

Now we know that $S_3$ is contained in exactly $3$ hyperplanes $H_1^j$ and 3 hyperplanes $H_2^j$. Moreover, we know that if $S_3 \subset H_1^j$ then $S_3 \nsubseteq H_2^j$ and vice versa. Note that among those three hyperplanes $H_1^j$'s we can always find two sharing the same $A_i$. Without loss of generality we can assume that $S_3 \subset H_1^1 \cap H_1^2$. We have two cases:

\begin{enumerate}
    \item $S_3 \subset H_2^5 \cap H_2^6$: we have $$S_1 \cup S_3 \subset H_1^1 \cap H_1^2 = \spann(R_1, R_2, R_3, A_1)$$ and $$S_2 \cup S_3 \subset H_2^5 \cap H_2^6 = \spann(T_1, T_2, T_3, A_1).$$ But then we get $$\dim(\spann(S_1,S_3)) = \dim(\spann(S_2,S_3)) = 3 = \dim(H_1^1 \cap H_1^2) = \dim(H_2^5 \cap H_2^6),$$ hence $$\spann(S_1,S_3) = H_1^1 \cap H_1^2 $$ and, similarly,  $$\spann(S_2,S_3) = H_2^5 \cap H_2^6.$$ Therefore, $$\{R_1, R_2, R_3, T_1, T_2, T_3, A_1\} \subset \spann(S_1, S_2, S_3) = H,$$ which is a contradiction. 
    \item $S_3 \nsubseteq H_2^5 \cap H_2^6$: without loss of generality let $S_3 \nsubseteq H_2^6$. This implies $S_3 \subset H_1^6$. But then we can consider $H_1^2$ and $H_1^6$ - we have $S_3 \subset H_2^5 \cap H_2^1$ and the remaining part is analogous to the previous case. 
\end{enumerate}

Thus, in all cases $S_3$ cannot be a subset of all $V(f_i)$'s. It follows that $S_3 \nsubseteq V(\pi_1^*q_1) \cap V(\pi_1^*q_2) \cap V(\pi_1^*q_3)$, which is a contradiction

Now, we need to consider the case when $H \cap S = S_1 \cup S_2$ and, without loss of generality, the intersection is not transversal along $S_1$. Note that $$H \cap H_1^1 \cap H_1^2 \cap H_2^1 \cap H_2^2 = (H \cap H_1^1 \cap H_1^2) \cap (H \cap H_2^1 \cap H_2^2) = S_1 \cap S_2 = \{P\}$$ hence $h_0, h_1^1, h_1^2, h_2^1, h_2^2$ is a basis of $V_p$. Note that $h_1^1, h_1^2, h_1^3$ are linearly dependent in $V_p$ since they intersect in a linear space of dimension 3 hence $h_1^3 = a_1 h_1^1 + a_2 h_1^2$ (and, analogously, $h_2^3 = b_1 h_2^1 + b_2 h_2^2$). We can multiply $h_1^1, h_1^2, h_2^1$ and $h_2^2$ by some constants so that $h_1^3 = h_1^1 + h_1^2$ and $h_2^3 = h_2^1 + h_2^2$. We can also multiply $h_3^1$ and $h_3^2$ by certain constants so that $f_1 = h_0h_3^1 - h_1^1h_2^1$ and $f_2 = h_0h_3^2 - h_1^2h_2^2$ contain $C_6$ in their zero loci. Let $f_3 = h_0h_3^3 - h_1^3h_2^3 = h_0h_3^3 - (h_1^1 + h_1^2)(h_2^1 + h_2^2)$. Note that $f_1, f_2, f_3$ is a basis of $\Symm^2 V_p \cap H^0(\mathcal{I}_{C_6}(2))$ since they are linearly independent in $\Symm^2 V_p$. This gives us a possibility to work with $f_i$ instead of $\pi_1^*(q_i)$.

Since $H$ and $S$ don't intersect transversally along $S_1$, for any $x \in S_1$ we have $\bigcap_{i = 1}^{3}T_xV(f_i) = T_xS \subset T_xH$. Fix some $g \in H^0(\mathcal{O}_{\mathbb{P}^5}(1))$ such that $g(p) \neq 0$. Then $\frac{h_0}{g}, \frac{h_1^1}{g}, \frac{h_1^2}{g}, \frac{h_2^1}{g}, \frac{h_2^2}{g}$ is a system of coordinates on the affine set $\mathbb{P}^5 \setminus V(g)$. Since we deal with cones with vertex $p$, i.e., none of the equations concerning us involves $g$, we can by a slight abuse of notation denote the coordinates on $\mathbb{P}^5 \setminus V(g)$ by $h_0, h_1^1, h_1^2, h_2^1, h_2^2$. Furthermore, we may identify $T_x(\mathbb{P}^5 \setminus V(g))$ with $\mathbb{P}^5 \setminus V(g)$ (since $\mathbb{P}^5 \setminus V(g) \simeq \mathbb{A}^5$). Also, $T_xH = V(h_0)$ for every $x \in H \setminus V(g)$.

Therefore, for any $x \in S_1 \setminus V(g)$ we have $h_0 = L_H \in \spann(L_{xV(f_1)}, L_{xV(f_2)}, L_{xV(f_3)})$, where $L_{xN}$ is a linear form defining the tangent space of $N$ at $x$. Let us compute $$L_{xV(f_i)} = \frac{\partial f_i}{\partial h_0}(x)\cdot h_0 + \frac{\partial f_i}{\partial h_1^1}(x)\cdot h_1^1 + \frac{\partial f_i}{\partial h_1^2}(x)\cdot h_1^2 + \frac{\partial f_i}{\partial h_2^1}(x)\cdot h_2^1 + \frac{\partial f_i}{\partial h_2^2}(x)\cdot h_2^2$$ for $x \in S_1 \setminus V(g)$.

\begin{enumerate}
    \item We have
    $$L_{xV(f_1)} = h_3^1(x)\cdot h_0 - h_2^1(x)\cdot h_1^1 - h_1^1(x)\cdot h_2^1 + h_0(x)\cdot \nabla h_3^1.$$ Now we can use the fact that for any $x \in S_1$ we have $h_0(x) = h_1^1(x) = h_1^2(x) = 0$ to get $$L_{xV(f_1)} = h_3^1(x)\cdot h_0 - h_2^1(x)\cdot h_1^1$$ or, in other words, 
    $$L_{xV(f_1)} = (h_3^1(x),-h_2^1(x),0,0,0).$$
    \item Analogously, $$L_{xV(f_2)} = (h_3^2(x),0,-h_2^2(x),0,0).$$
    \item Analogously, $$L_{xV(f_3)} = (h_3^3(x),-(h_2^1(x) + h_2^2(x)),-(h_2^1(x) + h_2^2(x)),0,0).$$
\end{enumerate} 

Let us write $h_3^i = \alpha_ih_2^1 + \beta_ih_2^2 + \mbox{other terms}$. Then we have $h_3^i(x) = \alpha_ih_2^1(x) + \beta_ih_2^2(x)$ for $x \in S_1$.

Note that $h_0 = (1,0,0,0,0) \in \spann(L_{xV(f_1)}, L_{xV(f_2)}, L_{xV(f_3)})$ for any $x \in S_1$ and for any $a, b \in \mathbb{C}$ we can find $x \in S_1$ such that $h_2^1(x) = a$ and $h_2^2(x) = b$. Let us take $x', x'' \in S_1$ such that $h_2^1(x') = h_2^2(x'') = 0$ and $h_2^2(x') = h_2^1(x'') = 1$. Then for some $D_1, D_2, D_3 \in \mathbb{C}$
\begin{equation*}
\begin{split}
   (1,0,0,0,0) &= D_1\cdot(h_3^1(x'),-h_2^1(x'),0,0,0) + D_2\cdot(h_3^2(x'),0,-h_2^2(x'),0,0) \\  
   & + D_3\cdot(h_3^3(x'),-(h_2^1(x') + h_2^2(x')),-(h_2^1(x') + h_2^2(x')),0,0)\\
   & = (D_1h_3^1(x') + D_2h_3^2(x') + D_3h_3^3(x'),-D_3,-D_2-D_3,0,0)
\end{split}
\end{equation*}
so $D_2 = D_3 = 0$ hence  $h_3^1(x') \neq 0 \Longrightarrow \beta_1 \neq 0$. Analogously, using $x''$, we obtain $\alpha_2 \neq 0$. Now consider $x \in S_1$ such that $h_2^1(x) = 1$ and $h_2^2(x) \neq 0$. Then for some $A,B,C \in \mathbb{C}$ $$(1,0,0,0,0) = (Ah_3^1(x) + Bh_3^2(x) + Ch_3^3(x), -A - C - Ch_2^2(x) ,-Bh_2^2(x) - C - Ch_2^2(x),0,0).$$ Hence $C \neq 0$ so we can redefine $A':= \frac{A}{C}$ and $B':=\frac{B}{C}$ and we have : $$\left(\frac{1}{C},0,0,0,0\right) = (A'h_3^1(x) + B'h_3^2(x) + h_3^3(x), -A' - 1 - h_2^2(x) , -B'h_2^2(x) - 1 - h_2^2(x),0,0).$$ Therefore, $A' = - 1 - t$ and $B' = -\frac{1 + t}{t}$, where $t = h_2^2(x)$. Also, we have $A'h_3^1(x) + B'h_3^2(x) + h_3^3(x) \neq 0$. Therefore $$A'h_3^1(x) + B'h_3^2(x) + h_3^3(x) = (-1 - t)(\alpha_1 + \beta_1t) - \frac{1 + t}{t}\cdot\left(\alpha_2 + \beta_2t\right) + \alpha_3 + \beta_3t = r(t)$$ has to be non-zero. But $$-t\cdot r(t) = \beta_1t^3 + (\beta_1 + \alpha_1 + \beta_2 - \beta_3)t^2 + (\alpha_1 + \beta_2 + \alpha_2 - \alpha_3)t + \alpha_2 $$ with $\beta_1 \neq 0 \neq \alpha_2$ has $3$ non-zero solutions $s_i$. Choosing $x_{s} \in S_1$ such that $h_2^1(x_s) = 1$ and $h_2(x_s) = s_1$ gives a contradiction with $h_0 \in \spann(L_{x_sV(f_1)}, L_{x_sV(f_2)}, L_{x_sV(f_3)})$. 

Thus,  $Q_1 \cap Q_2 \cap Q_3$ is a curve.  

\end{proof}

We already know that the intersection contains $C_p$ and two lines. By Bézout's theorem, $\deg(\bigcap_{i = 1}^3Q_i) = 2^3 = 8$ so $\bigcap_{i = 1}^3Q_i$ is a union of $C_p$ and two lines. We are ready to prove the main results of the paper.

\begin{theorem}

The ideal $I(C_p)$ of the curve $C_p$ is generated by three polynomials of degree $2$ and two polynomials of degree 3. 

\end{theorem}

\begin{proof}

From Proposition $\ref{p1}$ we know that the sequence $(q_1, q_2, q_3)$ is regular. This implies that there are no linear syzygies between $q_1, q_2, q_3$ (e.g. from the fact that the Koszul complex of $(q_1, q_2, q_3)$ is exact). Note that $x_iq_j \in H^0(\mathcal{I}_{C_p}(3))$ is a cubic containing $C_p$. Therefore, as a corollary, we have $$\dim U = 3 \cdot 5 = 15$$ where $U \subset H^0(\mathcal{I}_{C_p}(3))$ is a vector space generated by all possible $x_iq_j$. Thus, by Proposition \ref{prop3}, there are two linearly independent cubics $F_1 = V(f_1)$ and $F_2 = V(f_2)$ with $f_1, f_2 \subset H^0(\mathcal{I}_{C_p}(3))$ not in $U$ such that $$\spann(f_1, f_2) \oplus U = H^0(\mathcal{I}_{C_p}(3)).$$ Combining this with the first part of Proposition \ref{k} we get the result.

\end{proof}

\begin{theorem}

The ideal $I(C_{pq})$ of the curve $C_{pq}$ is generated by two polynomials of degree $3$ and three polynomials of degree 4.

\end{theorem}

\begin{proof}

From Proposition \ref{prprpr} we know that there are two linearly independent cubic polynomials in $I(C_{pq})$, which we will denote by $g_1$ and $g_2$. Let $G_1 = V(g_1)$ and $G_2 = V(g_2)$. We will show that $G_1 \cap G_2$ is one-dimensional. Suppose it is not. Then $g_1$ and $g_2$ have a common factor. By Bézout's theorem, if the factor has degree one, then the intersection is a sum of a plane and a curve of degree 4. If the common factor has degree 2, then the intersection is a sum of surface of degree 2 and a line. The former case is impossible since $C_{pq}$ is non-degenerate of degree 6, $C_{pq} \subset G_1 \cap G_2$ and $C_{pq}$ is not a plane curve. In the latter case, $C_{pq} \subset G_1 \cap G_2 = Q \cup l$,  where $l$ is a line and $Q = V(f)$ for some homogeneous polynomial $f$ of degree two, therefore $C_{pq} \subset Q$. Then $x_0f, x_1f, x_2f, x_3f \in H^0(I(C_{pq}(3))$ are cubics containing $C_{pq}$. They are linearly independent since if for some $\alpha_0,\ldots, \alpha_3$ we have $$\alpha_0x_0f + \alpha_1x_1f + \alpha_2x_2f + \alpha_3x_3f = 0$$ then $$(\alpha_0x_0 + \alpha_1x_1 +  \alpha_2x_2 +  \alpha_3x_3)f = 0$$ implying $\alpha_0 = \alpha_1 = \alpha_2 = \alpha_3 = 0$. But $H^0(I(C_{pq}(3))$ is two-dimensional - contradiction. Thus, $(g_1, g_2)$ is a regular sequence and there are no linear syzygies between $g_1$ and $g_2$.

Now, arguing as before, we see that there are $h^0(\mathcal{I}_{C_{pq}}(4)) - 4\cdot h^0(\mathcal{I}_{C_{pq}}(3)) = 11 - 8 = 3$ linearly independent quartics in $H^0(\mathcal{I}_{C_{pq}}(4))$, which do not come from $g_1$ and $g_2$. Combining this with the second part of Proposition \ref{k} we get the result.   

\end{proof}

\begin{remark} For the sake of completeness, let us study the image of the curve $C_{pq}$ under the generic projection onto $\mathbb{P}^2$. From \cite[main Theorem]{Lan} it follows that $\Sec{C_{pq}} = \mathbb{P}^3$. Thus, from the proof of \cite[Theorem 3.10]{Har} we see that there is an open subset  $U \subset \mathbb{P}^3$ such that the projection from each $O \in U$ is a birational morphism from $C_{pq}$ to its image in $\mathbb{P}^2$, and the image has at most nodes for singularities. Using the genus-degree formula we conclude that in this case the image of $C_{pq}$ has arithmetic genus 10, hence it has exactly 9 nodes.

\end{remark}

\section*{Acknowledgments}

The results presented in this article are a part of the author's Bachelor Thesis at Jagiellonian University in Kraków. I am extremely thankful to my supervisor, Michał Farnik, for the insightful discussions, constant support and valuable guidance throughout my studies. I would like to thank Paweł Borówka for the advice and feedback on the article and Grzegorz Kapustka for suggesting the topic. The author has been supported by the Polish National Science Center project number 2019/35/ST1/0238.

\bibliographystyle{alpha}

\end{document}